\documentclass[11pt,  reqno]{amsart}
\usepackage{latexsym,amsmath, amssymb, eucal, amscd, amstext, enumerate, mathrsfs, yfonts, amsfonts,comment, verbatim, color}
\usepackage[plainpages=false,colorlinks,pdfpagemode=None,bookmarksopen,linkcolor=blue,citecolor=blue,urlcolor=blue]{hyperref}
\usepackage[noadjust]{cite}
\usepackage{graphicx}
\graphicspath{{figures/}}
\usepackage[font=normal,labelfont=bf]{caption}

\def\tto{\;{\lower 1pt \hbox{$\rightarrow$}}\kern -10pt
\hbox{\raise 2pt \hbox{$\rightarrow$}}\;}

\newtheorem{Theorem}{Theorem}[section]
\newtheorem{Proposition}[Theorem]{Proposition}
\newtheorem{Lemma}[Theorem]{Lemma}

\theoremstyle{definition}

\theoremstyle{remark}

\newtheorem{question}[Theorem]{Question}

\numberwithin{equation}{section}

\begin{document}
\title[Bregman nonexpansive type actions]{Bregman nonexpansive type actions of semitopological semigroups}

\author[Muoi]{Bui Ngoc Muoi}
\address[Bui Ngoc Muoi]{Department of Mathematics, Hanoi Pedagogical University 2, Vinh Phuc, Vietnam, and Department of Applied Mathematics, National Sun Yat-sen University, Kaohsiung, 80424, Taiwan.}
\email{\tt muoisp2@gmail.com}

\author[Wong]{Ngai-Ching Wong}
\address[Ngai-Ching Wong]{Department of Applied Mathematics, National Sun Yat-sen University, Kaohsiung, 80424, Taiwan.}
\email{\tt wong@math.nsysu.edu.tw}
%\email{muoisp2@gmail.com, wong@math.nsysu.edu.tw}

\thanks{This research is supported by the Taiwan MOST grant 108-2115-M-110-004-MY2.}

\thanks{Corresponding author: Ngai-Ching Wong, E-mail: wong@math.nsysu.edu.tw}
\date{}% ver. July 23, 2020, submitted to JNCA
%\thanks{This research is supported by the Taiwan MOST grant 108-2115-M-110-004-MY2.}
\dedicatory{Dedicated to Professor Anthony To-Ming Lau on the occasion of his retirement}

\begin{abstract}
Let $S$ be a semitopological semigroup, and let $C$ be a nonempty closed convex subset of a reflexive Banach space. Under some amenability conditions on $S$,
we provide existence results of fixed points for several Bregman nonexpansive type actions $S\times C\to C$,
$(s,x)\mapsto T_s x$, of $S$ on $C$.
The mappings $T_s$ we discuss include those being Bregman generalized hybrid, Bregman nonspreading, and Bregman left asymptotically nonexpansive.
\end{abstract}

\keywords{Semitopological semigroups, left invariant means,
	fixed points, Bregman generalized hybrid maps, Bregman nonspreading maps, Bregman left asymptotic nonexpansive maps.}

\subjclass[2010]{Primary 47H10; Secondary 47H20, 47H09.}

\maketitle

\section{Introduction}

Let $S$ be a \emph{semitopological semigroup}, i.e., $S$ is a semigroup with a (Hausdorff) topology such that for each fixed $t\in S$, the mappings $s\rightarrow ts$ and $s\rightarrow st$ are continuous.
An \emph{action} of $S$ on a  nonempty set $C$ is a mapping of $S\times C$ into $C$, denoted by
$(s,x) \mapsto T_sx$, such that $T_{st}x=T_s(T_tx)$ for all $s, t\in S$ and $x\in C$. A point $x_0\in C$ is called
a \emph{common fixed point} for $S$ if $T_sx_0=x_0$ for all $s\in S$.

Let $X$ be a left  translation invariant normed vector subspace of $\ell^\infty(S)$, i.e., $l_sf\in X$ for all $s\in S, f\in X$,
where the \emph{left translation} $l_sf$ is given by $l_sf(t)=f(st), \forall t\in S$.  Assume
also that $X$ contains all the constant functions.
For example, the Banach algebra $\operatorname{CB(S)}$ of continuous and bounded functions on $S$ is a
left translation invariant subspace of $\ell^\infty(S)$
containing constants.
 A bounded linear functional $\mu$ on $X$ is called a \emph{mean} if $\|\mu\|=\mu(1)=1$.
  A mean $\mu$ is \emph{left invariant}, or a $\operatorname{LIM}$ in short, if $\mu(f)=\mu(l_sf)$ for all $s\in S$ and $f\in X$.
 A mean $\mu$ is called \emph{multiplicative} if $X$ is a
 subalgebra of  $\ell^\infty(S)$ and $\mu(fg)=\mu(f)\mu(g), \forall f,g\in X$.
We call $S$ \emph{left amenable} (resp.\ \emph{extremely left amenable}) if $\operatorname{CB(S)}$ has a left invariant mean
(resp.\ multiplicative left invariant mean).

It is a classical result of Mitchell \cite{Mit68} that %$\operatorname{CB(S)}$ has a multiplicative left invariant mean
$S$ is extremely left amenable if and only if every continuous action of $S$ on a compact set has a common fixed point.
When $S$ has only a non-multiplicative left invariant mean instead, however,
we need some nonexpansiveness of the action to guarantee a fixed point.

Lau and Takahashi \cite{LauTaka95} considered
\emph{left asymptotically nonexpansive} actions, while
Lau and Zhang \cite{LauZhang16} considered \emph{generalized hybrid} actions,  of a left amenable semigroup  $S$ on a nonempty set $C$ in
Hilbert space. They showed that such an action always has a common fixed point whenever there exists a bounded orbit
$$
O_c:=\{T_sc: s\in S\}
$$
of some $c$ in $C$. On the other hand,
we  studied norm-nonexpansive type actions on Banach spaces in \cite{MuoiWong2020}, and
metric- and seminorm-nonexpansive type actions on  Fr\'{e}chet and general locally convex spaces in \cite{MuoiWong2020JNCA}.

In this paper, we consider
nonexpansive type actions with respect to {Bregman distances} on reflexive Banach spaces with bounded orbits.
We note that Bregman distances (also called Bregman divergence),
though not being symmetric or satisfying the triangle inequality in general, are recently popular and useful in the
quantum information theory (see, e.g., \cite{CGLMVW17}).

Fix a {strictly convex} and {G\^{a}teaux differentiable} function $g: U\to\mathbb{R}$ defined
on an open set $U$ in a Banach space $E$ (see section \ref{section2} for
definitions and notations). The Bregman distance (see, e.g., \cite{Bregman67}) $D_g$ on $U$ is defined by
\begin{equation*}
D_g(x,y)=g(x)-g(y)-\left\langle x-y,\nabla g(y)\right\rangle,\; \forall x,y\in U.
\end{equation*}

\noindent
Let $S$ be a semitopological semigroup.
Let $C$ be a nonempty subset of $U$.
An action $(s,x) \mapsto T_sx$ of $S$ on $C$ is called
\begin{itemize}
	\item \textit{Bregman nonexpansive} if $$D_g(T_sx,T_sy)\leq D_g(x,y),\quad \forall x,y\in C,  \forall s\in S;$$
	\item \textit{Bregman left asymptotically nonexpansive} if for given $\varepsilon>0$ and $y\in C$, there exists an $s_\epsilon$ in $S$ depending on $\varepsilon$ and also on $y$, such that
%for the left ideal $J=St$ of $S$ we have
	\begin{equation}\label{BregLeftAsym}
	D_g(T_{ss_\epsilon}x,T_{ss_\epsilon}y)\leq D_g(x,y)+\varepsilon,\quad \forall s\in S, \forall x\in C;
	\end{equation}
	\item \textit{Bregman nonspreading} \cite{NWY2014} if \begin{equation}\label{BregNonspred}
	D_g(T_sx,T_sy)+D_g(T_sy,T_sx)\leq D_g(T_sx,y)+D_g(T_sy,x),\quad \forall x,y\in C, \forall s\in S;
	\end{equation}
	\item \textit{Bregman generalized hybrid} \cite{KJHH12} if there are real numbers $\alpha, \beta$ such that
	\begin{equation}\label{BregHybrid}
	\alpha D_g(T_sx,T_sy)+(1-\alpha)D_g(x,T_sy)\leq\beta D_g(T_sx,y)+(1-\beta)D_g(x,y),\quad \forall x,y\in C,  \forall s\in S.	
	\end{equation}
\end{itemize}

 It is plain that nonexpansive, left asymptotically nonexpansive \cite{LauTaka95}, nonspreading \cite{kt1,KohTaka08,twy}
 and generalized hybrid \cite{tak1} maps defined on Hilbert spaces are exactly those Bregman nonexpansive,
 Bregman left asymptotically nonexpansive, Bregman nonspreading and Bregman generalized hybrid maps
 with respect to the Bregman distance $D_g$ with $g(x)=\|x\|^2$.
  We also note that all Bregman nonexpansive mappings are Bregman left asymptotically nonexpansive and Bregman generalized hybrid,
 while Bregman nonspreading might be neither Bregman nonexpansive nor continuous; see, e.g., {\cite[example 1]{NWY2014}}.

We will show that any action  $(s,x)\mapsto T_s x$ of a left amenable semitopological semigroup $S$
 on a nonempty closed and convex subset $C$ of a reflexive Banach space $E$ has a common
fixed point, provided that all $T_s$ carry any one of the above Bregman nonexpansiveness and there is a bounded orbit $O_c$
of some point $c\in C$.

Here is a sketch of our reasoning.
 Let $X$ be a left translation invariant subspace of $\ell^\infty(S)$ containing  constants and all coordinate functions
  $s\mapsto \left\langle T_sc,x^*\right\rangle$ with $x^*\in E^*$.
  Let $\mu$ be a left invariant mean on $X$. Then for the bounded linear functional $x^*\mapsto \mu_s \left\langle T_sc,x^*\right\rangle$,
  there exists $z_\mu\in E=E^{**}$ such that
  $\left\langle z_\mu,x^*\right\rangle=\mu_s \left\langle T_sc,x^*\right\rangle$.
  We  pretend for a moment that $\mu$ is a probability measure on $S$.  Then we could write $$z_\mu=\int_S T_s c\, d\mu(s).$$
Note that the point $z_\mu\in C$ if $C$ is closed and convex.
For each $t\in S$, we would have
  $$
  T_tz_\mu=T_t\int_ST_sc\, d\mu(s)=\int_ST_{ts}c\, d\mu(s)=\int_ST_sc\, d\mu(s)=z_\mu,
  $$
since $\mu$ is left invariant.  In other words,
 $z_\mu$ is a candidate of the common fixed point for $S$. In the following, we will verify this claim.

 In section \ref{section2}, we describe briefly the Bregman distances on Banach spaces and their properties.
With respect to these distances, in section \ref{section3}, we study various Bregman type nonexpansive actions of a left
 amenable semitopological semigroup on a nonempty closed convex subset of a reflexive Banach space.
 We show that such an action has a common fixed point, if there is a bounded orbit and the left translation
 invariant function space $X$ generated by
 the action on
  the orbit  has a $\operatorname{LIM}$. The uniqueness of the fixed point is also discussed.
In section \ref{sectionBregSubspace}, we study the problem when the underlying function space $X$ has a $\operatorname{LIM}$ by embedding $X$ into some classical function spaces on $S$.
Finally, in section \ref{problems} we discuss some possible resolves for the case when we do not have any $\operatorname{LIM}$ in stock.

\section{Bregman distances}\label{section2}

Let $U$ be a nonempty open set in a Banach space $E$. A function $g: U\to\mathbb{R}$ is said to be \emph{G\^{a}teaux differentiable} at $y$ if there
is a bounded linear functional $\nabla g(y)$ in $E^*$, called the \emph{gradient} of $g$ at $y$, such that
$$
\left\langle x,\nabla g(y)\right\rangle=\displaystyle\lim_{t\rightarrow 0}\frac{g(y+tx)-g(y)}{t},\quad \forall x\in E.
$$
We call $g$ \textit{Fr\'{e}chet differentiable} at $y$ if for each given $\varepsilon>0$, there exists $\delta>0$ such that
$$
|g(x)-g(y)-\left\langle x-y,\nabla g(y)\right\rangle|\leq\varepsilon \|x-y\|\quad \mbox{ whenever } \|x-y\|\leq \delta.
$$
We call $g$ \emph{strictly convex} if $g(\alpha x+(1-\alpha)y)<\alpha g(x)+(1-\alpha)g(y)$ for all distinct $x,y\in U$ and $\alpha\in (0,1)$.

Let $g: U\to\mathbb{R}$ be a strictly convex and G\^{a}teaux differentiable function on a  nonempty open set $U$ in a Banach space $E$. The
\emph{Bregman distance} $D_g$ on $U$ is defined by
\begin{equation}\label{defBregDis}
D_g(x,y)=g(x)-g(y)-\left\langle x-y,\nabla g(y)\right\rangle,\quad \forall x,y\in U.
\end{equation}
It is known that if a strictly convex function $g$ is
 G\^{a}teaux differentiable then its gradient $\nabla g$ is norm-to-weak* continuous  (see, e.g., \cite[Proposition 1.1.10]{ButIu00}) and
 $
 D_g(x,y)\geq 0,\ \forall x,y\in U;
 $
 the equality holds exactly when $x=y$.
 If $g$ is Fr\'{e}chet differentiable then $\nabla g$ is norm-to-norm continuous (see, e.g., \cite[page 508]{kt3}).

Some Bregman distances between positive definite matrices used in quantum information theory  associated with the Bregman function $g(A)=\operatorname{trace}(f(A))$ follows:

\begin{itemize}
		\item \textbf{classical divergence:} $D_g(A,B)=\operatorname{trace}(A^2)+\operatorname{trace}(B^2)-2\operatorname{trace}(BA)$ while $f(x)=x^2$;
		\item \textbf{Umegaki relative entropy:} $D_g(A,B)=\operatorname{trace}\left[A(\log A-\log B)\right]$ while $f(x)=x\log x$;
		\item \textbf{Quantum divergence:} $D_g(A,B)=\|\sqrt{A}-\sqrt{B}\|^2_2$\; while $f(x)=(\sqrt{x}-1)^2$, where $\|\cdot\|_2$ is
the Hilbert-Schmidt norm of matrices.
\end{itemize}
In the following demonstrations, $E=\mathbb{R}$, $U=(0,+\infty)$ and thus $g=f$.
\bigskip
%\vspace{1cm}
\begin{center}
		\includegraphics[scale=0.17]{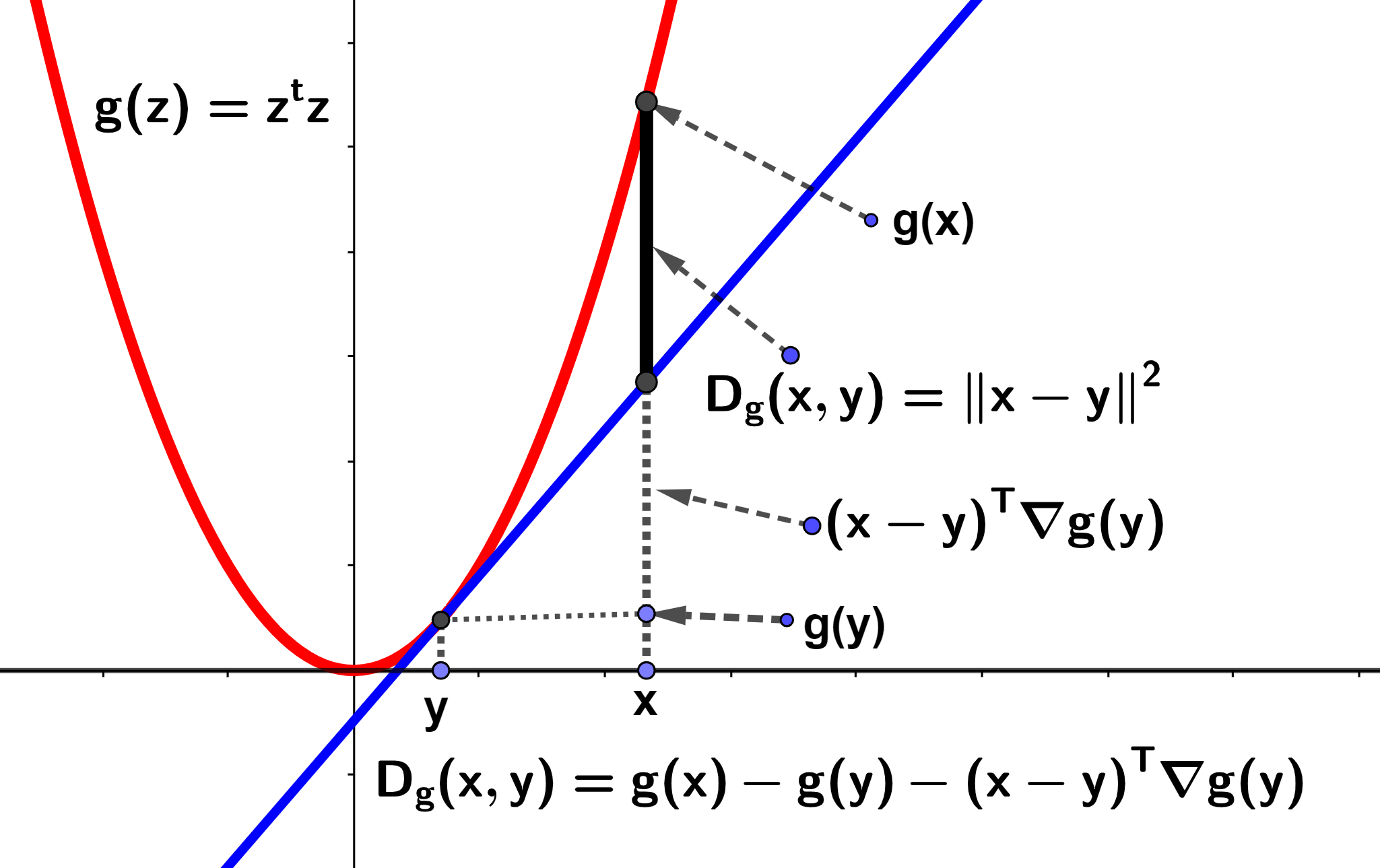}
	\captionof*{figure}{ \textbf{Figure 1:} The Bregman classical divergence is exactly the square Hilbert space distance.}
\end{center}

\bigskip

\begin{center}
	{	\hspace{0.3cm}
		\centering
		\includegraphics[scale=0.12]{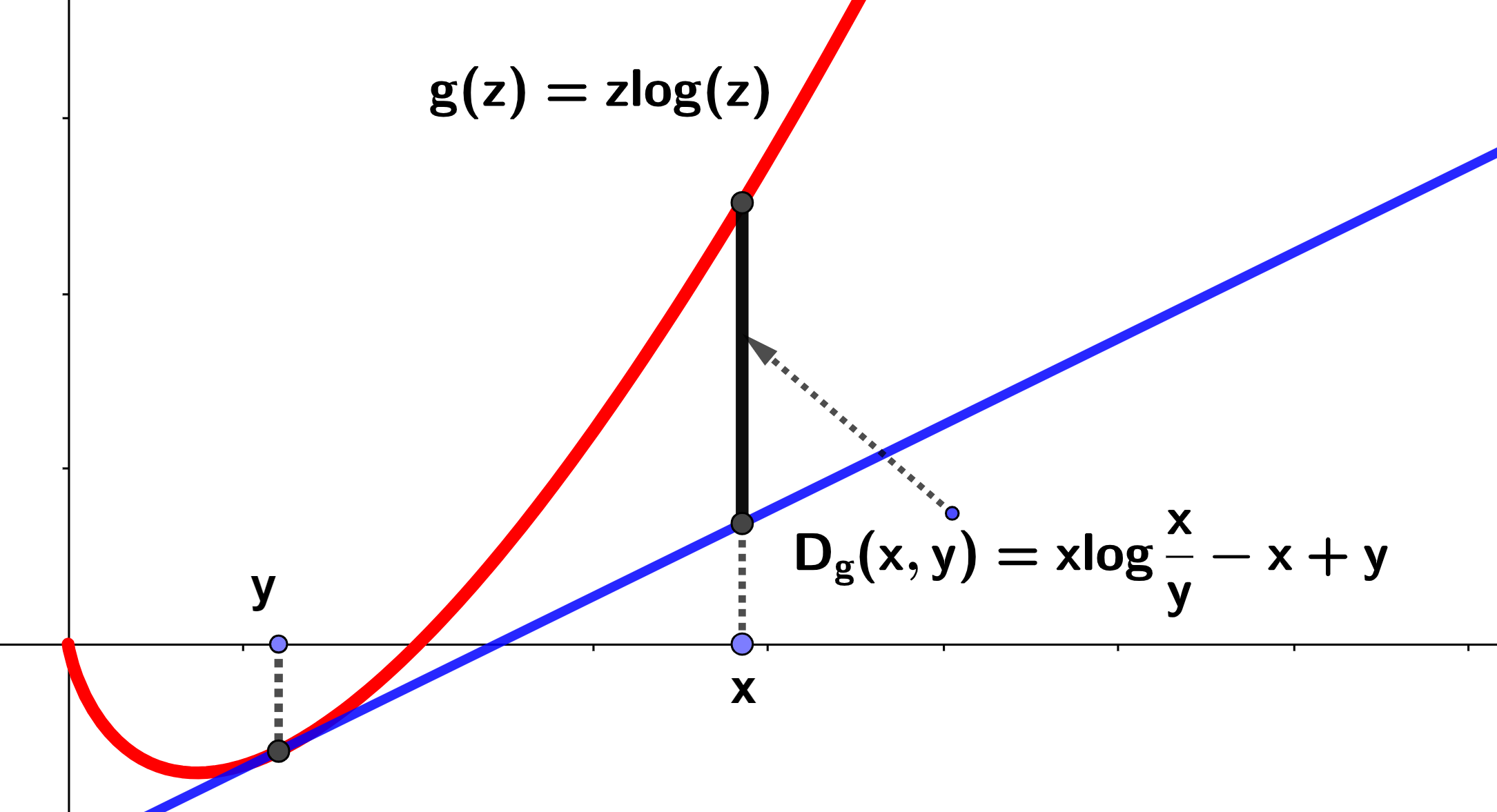} 	
	}
	{ \hspace{0.3cm}
		\centering
		\includegraphics[scale=0.133]{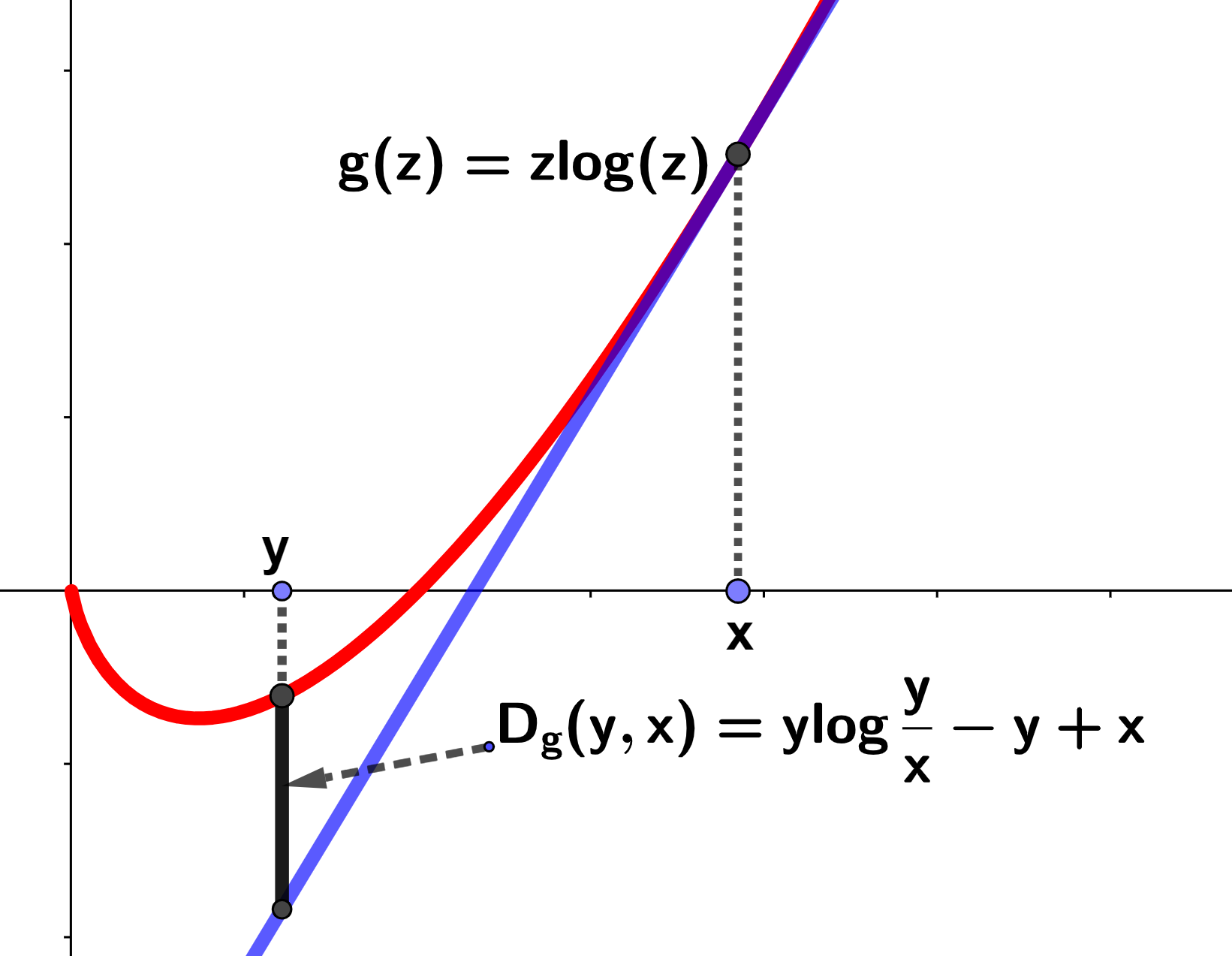}
	}
	\captionof*{figure}{ \textbf{Figures 2, 3:} Umegaki relative entropy is not symmetric, $D_g(x,y)>D_g(y,x)$.}
\end{center}

\bigskip

Although the Bregman distances $D_g(x,y)$ have deficiencies such as not necessarily being symmetric, not necessarily satisfying
the triangle inequality and not necessarily translation invariant, they do carry the \textit{Bregman-Opial property}.
That is, for any weakly convergent sequence $x_n\to  x$ in $U$, we have
$$
\displaystyle\limsup_{n\rightarrow\infty}D_g(x_n,x)<\limsup_{n\rightarrow\infty}D_g(x_n,y),\quad \forall y\in U\setminus\{x\}.
$$
For the function $g(x)=\|x\|^2$ on a Hilbert space, this reduces to the famous Opial property.
However, unlike the Bregman-Opial property, the Opial property fails to hold  in the general Banach space setting.
See, e.g., \cite{hjkh}.

We call a function $g: E\to\mathbb{R}$ \textit{strongly coercive} if $\displaystyle\lim_{\|x_n\|\to\infty}\frac{g(x_n)}{\|x_n\|}=+\infty$,
call $g$ \emph{locally bounded} if $g$ is bounded on bounded sets, and call
$g$  a \emph{Bregman function} \cite{ButIu00} if it satisfies the following conditions.
\begin{itemize}
	\item [(1)] $g$ is continuous, strictly convex and G\^{a}teaux differentiable.
	\item [(2)] The set $\{y\in E: D_g(x,y)\leq r\}$ is bounded for all $x\in E$ and all $r>0$.
\end{itemize}

\begin{Lemma}[see {\cite[page 70]{ButIu00}}]\label{LemmaBregProj}
	Let $C$ be a nonempty, closed and convex subset of $E$. Let $g$ be a strongly coercive Bregman function on $E$, then for each $x\in E$, there exists a unique point  $\hat{x}\in C$ such that $$D_g(\hat{x},x)=\displaystyle\min_{y\in C}D_g(y,x).$$   In this case,  $D_g(y,\hat{x})\leq D_g(y,\hat{x})+D_g(\hat{x},x)\leq D_g(y,x)$ for all $y\in C.$
\end{Lemma}
We call $\hat{x}$ the \emph{Bregman projection} of $x$ on $C$ and denote it by $P_C^g(x)$.

\section{Fixed point properties for Bregman nonexpansive type actions}\label{section3}

In this section, we always consider an action $S\times C\to C$, denoted by $(s,x)\mapsto T_s x$,
of a semitopological semigroup $S$ on a nonempty closed convex subset $C$ of a reflexive Banach space $E$.
Moreover, we always assume $g:E\to \mathbb{R}$ is a strongly coercive and locally bounded Bregman function.

Assume that there exists a point $c\in C$ such that the orbit $O_c=\{T_sc: s\in S\}$ is   bounded.
Let $X$ be the intersection of all left translation invariant subspaces of $\ell^\infty(S)$, which contains
all the constant functions and all the coordinate functions
\begin{equation}\label{BregCoefunction}
s\to \left\langle T_sc,x^*\right\rangle
\quad \mbox{ and }\quad s\to D_g(T_sc,x),\quad \forall x\in E,\ x^*\in E^*
\end{equation}
generated by the orbit $O_c$.
We call $X$ the \emph{Bregman subspace} associated to the action on the bounded orbit $O_c$. Note that the local boundedness of $g$ ensures that the functions in \eqref{BregCoefunction} are bounded on $S$.
For any mean $\mu$ on $X$, we consider the bounded linear functional of the dual space $E^*$ of $E$
defined  by $x^*\mapsto \mu_s\left\langle T_sc,x^*\right\rangle$.
We shall write $\mu_s(f(s))$ instead of $\mu(f)$ for the value of $\mu$ at the function $s\mapsto f(s)$ in $s$.
Since $E$ is reflexive, there exists a unique $z_\mu\in E$ such that
\begin{equation}\label{PsiFunctional}
\mu_s\left\langle T_sc,x^*\right\rangle=\left\langle z_\mu,x^*\right\rangle, \quad \forall x^*\in E^*.
\end{equation}
We call $z_\mu$ the \emph{$\mu-$barycenter} of the bounded orbit $O_c$.

We are going to show that $z_\mu$ is a common fixed point of the action, provided
that some Bregman type nonexpansiveness is assumed.
As an intermediate step, we will show that $z_\mu$ is a Bregman attractive point, namely it belongs to
 the set of \emph{Bregman attractive points}  defined by
\begin{equation}\label{BregAttrac}
A^g_C(S)=\{x\in E: D_g(x,T_sy)\leq D_g(x,y), \forall y\in C, \forall s\in S\}.
\end{equation}
Clearly, $z_\mu$ is a common fixed point of the action exactly when $z_\mu\in A^g_C(S)\cap C$.

The following lemma can be deduced from \cite[page 209]{TWY14}, we give a different proof since its arguments will be vital for some later parts.
\begin{Lemma}\label{Lemmazmu}
	If $C$ is closed and convex then $z_\mu\in C$.
\end{Lemma}
\begin{proof}
	From \eqref{defBregDis}, the following Bregman \textit{three-point identity} holds for any $x,y,z\in E$,
	\begin{equation}\label{BregThreeProint}
	D_g(x,z)=D_g(x,y)+D_g(y,z)+\left\langle x-y,\nabla g(y)-\nabla g(z)\right\rangle.
	\end{equation}
	Therefore, for each $x\in E$ we have $$D_g(T_sc,x)-D_g(T_sc,z_\mu)=D_g(z_\mu,x)+\left\langle T_sc-z_\mu,\nabla g(z_\mu)-\nabla g(x)\right\rangle.$$
	Taking $\mu_s$ on both sides,
	\begin{equation}\label{eq222}
	\begin{array}{rl}
	\mu_s D_g(T_sc,x)-\mu_s D_g(T_sc,z_\mu)&=D_g(z_\mu,x)+\mu_s \left\langle T_sc-z_\mu,\nabla g(z_\mu)-\nabla g(x)\right\rangle\\
	&=D_g(z_\mu,x)+\left\langle z_\mu-z_\mu,\nabla g(z_\mu)-\nabla  g(x)\right\rangle\\
	&=D_g(z_\mu,x)\geq 0.
	\end{array}
	\end{equation}
	Since $D_g(z_\mu,x)=0$ exactly when $x=z_\mu$, we have
	\begin{equation}\label{eq2}
	M_\mu:=\{x\in E: \mu_sD_g(T_sc,x)=\displaystyle\inf_{y\in E}\mu_sD_g(T_sc,y)\}=\{z_\mu\}.
	\end{equation}

	Let $P^g_C(z_\mu)$ be the Bregman projection of $z_\mu$ on $C$. By Lemma \ref{LemmaBregProj}, since $T_sc\in C$ we have
	$D_g(T_{s}c,P^g_C(z_\mu))\leq D_g(T_{s}c,z_\mu).$
	Therefore, $\mu_s D_g(T_{s}c,P^g_C(z_\mu))\leq \mu_s D_g(T_{s}c,z_\mu)$.
	Since $z_\mu$ is the unique element of $M_\mu$, we have $P^g_C(z_\mu)=z_\mu$. Hence, $z_\mu\in C$.	
\end{proof}

\begin{Theorem}\label{mainThm}
	Let $C$ be a closed convex subset of a reflexive Banach space $E$. Let $(s,x)\mapsto T_s x$ be a
Bregman generalized hybrid action of a
semitopological semigroup $S$ on $C$ with a bounded orbit $O_c$.
If the Bregman subspace $X$ associated to $O_c$ has a left invariant mean $\mu$,
then the $\mu-$barycenter $z_\mu$ of $O_c$ is a common fixed point of $S$.
\end{Theorem}
\begin{proof}
If suffices to show that $z_\mu$ is a {Bregman attractive point} of the action; see \eqref{BregAttrac}.
Indeed, for each $t\in S$, from the three-point identity \eqref{BregThreeProint}, we have
$$
D_g(z_\mu,T_tx)-D_g(z_\mu,x)=D_g(x,T_tx)+\left\langle z_\mu-x,\nabla g(x)-\nabla g(T_tx)\right\rangle.
$$
From the definition of $z_\mu$ we have
	\begin{equation*}
	\begin{array}{rl}
	&\phantom{=}\ D_g(x,T_tx)+\left\langle z_\mu-x,\nabla g(x)-\nabla g(T_tx)\right\rangle\\
	&=D_g(x,T_tx)+\mu_s\left\langle T_sc-x,\nabla g(x)-\nabla g(T_tx)\right\rangle\\
	&=D_g(x,T_tx)+\mu_s\left\langle T_{ts}c-x,\nabla g(x)-\nabla g(T_tx)\right\rangle,\; \mbox{since } \mu \mbox{ is a }\operatorname{LIM},\\
	&=D_g(x,T_tx)+\mu_s\left[D_g(T_{ts}c,T_tx) -D_g(T_{ts}c,x)-D_g(x,T_tx)\right]\\
	&=\mu_sD_g(T_{ts}c,T_tx)-\mu_sD_g(T_{ts}c,x)\\
	&=\mu_sD_g(T_{ts}c,T_tx)-\mu_sD_g(T_{s}c,x),\; \mbox{since }\mu \mbox{ is a }\operatorname{LIM}.
	\end{array}
	\end{equation*}
Because the action is Bregman generalized hybrid, there exists $\alpha, \beta\in\mathbb{R}$ satisfying \eqref{BregHybrid}. Thus, by following an idea in \cite[Lemma 5.1]{LauZhang16},
	\begin{equation}\label{eq1}
	\begin{array}{rl}
	\mu_sD_g(T_{ts}c,T_tx)&=\mu_s\left[\alpha D_g(T_{ts}c,T_tx)+(1-\alpha)D_g(T_{ts}c,T_tx)\right]\\
	&=\mu_s\left[\alpha D_g(T_{t}T_{s}c,T_tx)+(1-\alpha)D_g(T_{s}c,T_tx)\right],\; \mbox{since }\mu \mbox{ is a }\operatorname{LIM},\\
	&\leq \mu_s\left[\beta D_g(T_{ts}c,x)+(1-\beta)D_g(T_sc,x)\right]\\
	&=\mu_s\left[\beta D_g(T_{s}c,x)+(1-\beta)D_g(T_sc,x)\right],\; \mbox{since }\mu \mbox{ is a }\operatorname{LIM},\\
	&=\mu_s D_g(T_sc,x).
	\end{array}
	\end{equation}
Therefore,
$$
D_g(z_\mu,T_tx)\leq D_g(z_\mu,x),\quad\forall t\in S,\ x\in C.
$$
In other words, $z_\mu$ is a Bregman attractive point.
	
	By Lemma \ref{Lemmazmu}, $z_\mu\in C$, and thus
$D_g(z_\mu,T_tz_\mu)\leq D_g(z_\mu,z_\mu)=0$.  Consequently, $T_tz_\mu=z_\mu$ for all $t\in S$.
    In other words, $z_\mu$ is a common fixed point of $S$.	
\end{proof}
In the following theorems, we need to assume, in addition, that
the Bregman subspace $X$ contains all the functions $s\to D_g(x,T_sc), \forall x\in C$.

\begin{Theorem}\label{mainThm2}
	Let $C$ be a closed convex subset of a reflexive Banach space $E$.
Let $(s,x)\mapsto T_s x$ be a Bregman nonspreading action of a semitopological semigroup $S$ with a bounded orbit $O_c$.
If the Bregman subspace $X$ associated to $O_c$ has a left invariant mean $\mu$,
then the $\mu-$barycenter $z_\mu$ of $O_c$ is a common fixed point of $S$.
\end{Theorem}
\begin{proof}
The proof goes exactly as in that of Theorem \ref{mainThm}, except for \eqref{eq1}.
In other words, it suffices to verify that
\begin{align}\label{eq:NSineq}
\mu_sD_g(T_{ts}c,T_tx)\leq\mu_s D_g(T_sc,x), \quad\forall t\in S,\ x\in C
\end{align}
Suppose it is not the case, and let $t_0\in S$ and $x_0\in C$ such that $\mu_sD_g(T_{t_0s}c,T_{t_0}x_0)>\mu_s D_g(T_sc,x_0)$. Since $T_{t_0}$ is Bregman nonspreading, by \eqref{BregNonspred} we have
 $$
 D_g(T_{t_0}T_sc,T_{t_0}x_0)+D_g(T_{t_0}x_0,T_{t_0}T_sc)\leq D_g(T_{t_0}T_sc,x_0)+D_g(T_{t_0}x_0,T_sc).
 $$
 By taking $\mu_s$ on both sides, we have
\begin{equation*}
\begin{array}{rl}
\mu_s D_g(T_{t_0s}c,T_{t_0}x_0)+\mu_s D_g(T_{t_0}x_0,T_{t_0s}c)&\leq \mu_s D_g(T_{t_0s}c,x_0)+\mu_s D_g(T_{t_0}x_0,T_sc)\\
&=\mu_s D_g(T_sc,x_0)+\mu_s D_g(T_{t_0}x_0,T_sc)\\
&< \mu_s D_g(T_{t_0s}c, T_{t_0}x_0)+\mu_s D_g(T_{t_0}x_0,T_sc).
\end{array}
\end{equation*}
Thus $\mu_s D_g(T_{t_0}x_0,T_{t_0s}c)<\mu_s D_g(T_{t_0}x_0,T_sc)$. However, this conflicts with the fact that $\mu$ is a $\operatorname{LIM}$.
This contradiction finishes the verification.
\end{proof}	

Recall that a vector subspace $X$ of $\ell^\infty(S)$ is called \emph{right translation invariant} if $r_sf\in X$ for all $f\in X$,
 where $r_sf(t)=f(ts), \forall t\in S$, is the \emph{right translation} of $f$.
A mean $\mu$ on a right translation invariant subspace
 is called \emph{right  invariant} if $\mu(r_sf)=\mu(f)$ for all
 $s\in S, f\in X$.
If $X$ is translation invariant, i.e., both left and right translation invariant,
we call a mean $\mu$ on $X$ an \textit{invariant mean} if it is both left and right invariant.
On the other hand, we call  $\mu$  \textit{faithful} if $f=0$ whenever $f\in X$, $f\geq 0$ and $\mu(f)=0$.

\begin{Theorem}\label{mainthm3}
Let $C$ be a closed convex subset of a reflexive Banach space $E$.
Let $(s,x)\mapsto T_s x$ be a Bregman left asymptotically nonexpansive action of a semitopological semigroup $S$ with a bounded orbit $O_c$.
If the Bregman subspace $X$ associated to $O_c$ is translation invariant and has a faithful invariant mean $\mu$,
then the $\mu-$barycenter $z_\mu$ of $O_c$ is a common fixed point of $S$.
\end{Theorem}

\begin{proof}
 By Lemma \ref*{Lemmazmu}, $z_\mu\in C$. Since the action is Bregman left asymptotically nonexpansive,
 by \eqref{BregLeftAsym} for each given $\varepsilon>0$  there exists $s_\varepsilon\in S$ such that
$$
D_g(T_{ss_\varepsilon}T_tc,T_{ss_\varepsilon}z_\mu)\leq D_g(T_tc,z_\mu)+\varepsilon,\quad 	\forall s, t\in S.
$$
Then
	\begin{equation*}
	\begin{array}{rl}
	\mu_t D_g(T_tc,T_{ss_\varepsilon}z_\mu)&=\mu_t D_g(T_{ss_\varepsilon}T_tc,T_{ss_\varepsilon}z_\mu)\\
	&\leq \mu_t D_g(T_tc,z_\mu)+\varepsilon,\quad	\forall s\in S.
	\end{array}
	\end{equation*}
On the other hand, from \eqref{eq222} we have $D_g(z_\mu,x)=\mu_t D_g(T_tc,x)-\mu_t D_g(T_tc,z_\mu)$ for all $x\in E$. Hence, putting $x=T_{ss_\varepsilon}z_\mu$ we get
	\begin{equation}\label{eq55}
	\begin{array}{rl}
	D_g(z_\mu, T_{ss_\varepsilon}z_\mu)&=\mu_t D_g(T_tc,T_{ss_\varepsilon}z_\mu)-\mu_t D_g(T_tc,z_\mu)\\
	&\leq \mu_t D_g(T_tc,z_\mu)+\varepsilon-\mu_t D_g(T_tc,z_\mu)=\varepsilon.
	\end{array}
	\end{equation}
	Since $\mu$ is right invariant, $\mu_s D_g(z_\mu, T_{s}z_\mu)=\mu_s D_g(z_\mu, T_{ss_\varepsilon}z_\mu)\leq\varepsilon$ for all
$\varepsilon>0$.
	Therefore $\mu_s D_g(z_\mu, T_{s}z_\mu)=0$. Since $\mu$ is faithful and $D_g(z_\mu, T_{s}z_\mu)\geq 0$ for all $s\in S$, we have
$D_g(z_\mu, T_{s}z_\mu)=0$. This implies $z_\mu=T_sz_\mu$ for all $s\in S$. Then $z_\mu$ is a common fixed point for $S$.
\end{proof}

We are now going to discuss the uniqueness of the fixed point. In fact, as shown in Propositions \ref{independence}
and \ref{ProposProject} below,
the common fixed point $z_\mu$, which is the $\mu-$barycenter of a bounded orbit $O_c$,  is the same for any choice of an
\emph{invariant} mean $\mu$.

\begin{Lemma}\label{LemmaInfSup}
Let $C$ be a closed convex subset of a reflexive Banach space $E$.
Let $(s,x)\mapsto T_s x$ be an action of a semitopological semigroup $S$ with a bounded orbit $O_c$.
Let the Bregman subspace $X$ associated to $O_c$ be translation invariant and have an  invariant mean $\mu$,
 and let $y$ be a common fixed point. Then
	 \begin{itemize}
	 	\item [(i)] $\displaystyle\sup_t\inf_sD_g(T_{ts}c,y)\leq\mu_tD_g(T_tc,y)\leq \displaystyle\inf_s\sup_tD_g(T_{ts}c,y)$ if the action is Bregman generalized hybrid or Bregman nonspreading;
	 	\item [(ii)] $\displaystyle\sup_t\inf_sD_g(T_{ts}c,y)=\mu_tD_g(T_tc,y)=\displaystyle\inf_{s}\sup_{t}D_g(T_{ts}c,y)$ if the action
is  Bregman left asymptotically nonexpansive.
	 \end{itemize}
\end{Lemma}
\begin{proof}
	We follow an idea in \cite{LauTaka95} which works for the Hilbert space.
	For each $f\in X$ and $s\in S$, by the right invariance of $\mu$, we have
$$
\mu_tD_g(T_tc,y)=\mu_tD_g(T_{ts}c,y)\leq\displaystyle\sup_{t\in S} D_g(T_{ts}c,y).
$$
Hence,
$$
\mu_tD_g(T_tc,y)\leq \displaystyle\inf_s\sup_tD_g(T_{ts}c,y).
$$
	
	(i) As shown in \eqref{eq1} and \eqref{eq:NSineq}, for each $t, u\in S$ we have
$$
\mu_s D_g(T_sc,y)\geq \mu_s D_g(T_{tu}T_sc,T_{tu}y)\geq\inf_s D_g(T_{ts}c,y).
$$
Thus $\mu_s D_g(T_sc,y)\geq \displaystyle\sup_t\inf_sD_g(T_{ts}c,y)$.	
	
	{(ii)} Since $y$ is a common fixed point and the action is  Bregman left asymptotically nonexpansive,
for any $\epsilon>0$ there is $u\in S$ such that
$$
D_g(T_{tus}c,y) = D_g(T_{tu}T_sc,T_{tu}y)\leq D_g(T_{s}c,y)+\epsilon, \quad \forall s,t\in S.
$$
Consequently,
	\begin{equation*}
	\begin{array}{rl}
\displaystyle\inf_{s'}\sup_t D_g(T_{ts'}c,y)&\leq\displaystyle\sup_t D_g(T_{tus}c,y)
%=\displaystyle\sup_t D_g(T_{tu}T_sc,T_{tu}y)
\leq\displaystyle\sup_{t}D_g(T_sc,y) + \epsilon
\\
&= D_g(T_sc,y) + \epsilon, \quad\forall s\in S.
	\end{array}
	\end{equation*}
Because $\epsilon >0$ is arbitrary, $\displaystyle\inf_{s}\sup_t D_g(T_{ts}c,y)\leq \mu_s D_g(T_sc,y)$. Hence
$\mu_tD_g(T_tc,y)=\displaystyle\inf_{s}\sup_{t}D_g(T_{ts}c,y)$.	

On the other hand, by the  Bregman left asymptotically nonexpansiveness again,
\begin{align*}
\displaystyle\sup_{t'}\inf_sD_g(T_{t's}c,y)
&\geq \inf_s D_g(T_{us}c,y)\geq \inf_s \mu_t D_g(T_{tu}T_{us} c, T_{tu}y) - \epsilon\\
&\geq \inf_s \mu_{t} D_g(T_{tuus} c, y)-\epsilon\\
%&= \inf_s \mu_{u} D_g(T_{uts} c,y)-\epsilon\\
&= \inf_s \mu_{t} D_g(T_{t} c,y)-\epsilon,\quad\text{since $\mu$ is right invariant,}\\
&= \mu_t D_g(T_t c, y)-\epsilon.
\end{align*}
Since $\epsilon >0$ is arbitrary, we have
$$
\displaystyle\sup_{t}\inf_sD_g(T_{ts}c,y)\leq
\displaystyle\inf_{s}\sup_{t}D_g(T_{ts}c,y)=\mu_tD_g(T_tc,y)\leq\displaystyle\sup_{t}\inf_sD_g(T_{ts}c,y).
$$
\end{proof}

\begin{Proposition}\label{independence}
Let $(s,x)\mapsto T_s x$ be an
action of a semitopological semigroup $S$ on a nonempty
 closed convex set $C$ in a reflexive Banach space $E$.
Let  the Bregman subspace $X$ associated to a bounded orbit $O_c$  is translation invariant.
Assume that the barycenters $z_\mu, z_\psi$ are   common fixed points of $S$ for   invariant means $\mu, \psi$ on $X$, respectively.
Assume that the
action is either
\begin{itemize}
	\item [(i)] Bregman left asymptotically nonexpansive, or
	\item [(ii)] Bregman generalized hybrid or Bregman nonspreading, such that
$$
\displaystyle\sup_t\inf_sD_g(T_{ts}c,z)= \displaystyle\inf_s\sup_tD_g(T_{ts}c,z) \quad\text{for $z=z_\mu$ or $z=z_\psi$}.
 %\displaystyle\sup_t\inf_sD_g(T_{ts}c,z_\psi)= \displaystyle\inf_s\sup_tD_g(T_{ts}c,z_\psi).
$$
\end{itemize}
 Then, $\mu_tD_g(T_tc,z_\mu)=\psi_tD_g(T_tc,z_\psi)$ and $z_\mu=z_\psi$.
\end{Proposition}

\begin{proof}
From Lemma \ref{LemmaInfSup} and the assumptions, in each case, we have $$\mu_tD_g(T_tc,z_\mu)=\displaystyle\inf_{s}\sup_{t}D_g(T_{ts}c,z_\mu)\quad \mbox{and}\quad \psi_tD_g(T_tc,z_\psi)=\displaystyle\inf_{s}\sup_{t}D_g(T_{ts}c,z_\psi).$$
Following an idea in \cite[Theorem 4.8]{LauTaka95}, we obtain from \eqref{eq2} that
\begin{equation*}
\begin{array}{rl}
\mu_tD_g(T_tc,z_\mu)&\leq\mu_tD_g(T_tc,z_\psi)= \mu_tD_g(T_{ts}c,z_\psi)\leq \displaystyle\inf_{s}\sup_{t}D_g(T_{ts}c,z_\psi)\\
&= \psi_tD_g(T_tc,z_\psi) \leq \psi_t D_g(T_tc,z_\mu)\leq\displaystyle\inf_{s}\sup_{t}D_g(T_{ts}c,z_\mu)\\
&=\mu_tD_g(T_tc,z_\mu).
\end{array}
\end{equation*}
Hence, $\mu_tD_g(T_tc,z_\mu)=\psi_tD_g(T_tc,z_\psi)$ and $\mu_tD_g(T_tc,z_\mu)=\mu_t D_g(T_tc,z_\psi)$.
Since $M_\mu$ contains the unique point $z_\mu$, we have $z_\mu=z_\psi$.
\end{proof}

The following proposition extends a result of Lau and Zhang \cite[Theorem 4.11]{LauZhang16} in which they considered norm nonexpansive
actions on Hilbert spaces.  It implies that the barycenter $z_\mu$ of a bounded orbit $O_c$ does not depends on the choice of
the invariant mean $\mu$, provided that $S$ is right reversible together with other assumptions.
Recall that a semitopological semigroup $S$ is called \textit{right reversible} if
the intersection $\overline{Sa}\cap\overline{Sb}$ of two closed left ideals is nonempty for every $a, b\in S$.
In this case, we can define a direct order on $S$ by letting $a\leq b$
if $b=a$ or $b\in\overline{Sa}$. It is easy to see that if $a\leq b$ then $a\leq tb$ for all $t\in S$.
In particular, $ts\to \infty$ whenever $s\to\infty$ for any fixed $t$ in $S$.

Recall that if the Bregman function $g$ is Fr\'{e}chet differentiable then $\triangledown g$ is norm-to-norm continuous.
\begin{Proposition}\label{ProposProject}
	Let $C$ be a nonempty closed convex subset of a reflexive Banach space $E$. Assume that
%a strongly coercive and locally bounded Bregman function
%$g$ is Fr\'{e}chet differentiable, and
$\nabla g$ is norm-to-norm continuous and
the Bregman distance $D_g$ is symmetric.
Let $(s,x)\mapsto T_s x$ be a Bregman generalized hybrid or Bregman nonspreading action
of a right reversible semitopological semigroup $S$ on $C$ with  a bounded orbit $O_c$.
 Assume that the Bregman subspace $X$ associated to $O_c$ is translation invariant and has
 an invariant mean $\mu$. Then $A^g_C(S)$ is nonempty, closed and convex. Moreover, $\displaystyle\lim_tP^g_{A^g_C(S)}(T_tc)=z_\mu$ in norm.
\end{Proposition}

\begin{proof}
It is shown in the proof of Theorems \ref{mainThm} and \ref{mainThm2}
 that $z_\mu\in A^g_C(S)$.	Hence $A^g_C(S)$ is nonempty. Furthermore,
 as shown in \cite[Lemmas 4.9 and 4.10]{LauZhang16}, see also \cite[Lemmas 3.2 and 4.4]{ENC2019},
 $A^g_C(S)$ is closed convex,
  and $P^g_{A^g_C(S)}(T_tc)$ converges strongly to some $u\in A^g_C(S)$.
Since $D_g$ is symmetric, i.e., $D_g(x,y)=D_g(y,x), \forall x,y\in E$, by Lemma \ref{LemmaBregProj}, for each $t, s\in S$ we have
$$D_g(T_{ts}c,P^g_{A^g_C(S)}(T_{ts}c))+D_g(P^g_{A^g_C(S)}(T_{ts}c),y)\leq D_g(T_{ts}c,y).$$
By the three-point identity \eqref{BregThreeProint}, we have
$$
\left\langle T_{ts}c-P^g_{A^g_C(S)}(T_{ts}c),\nabla g(P^g_{A^g_C(S)}(T_{ts}c))-\nabla g(y) \right\rangle\geq 0,
\quad \forall y\in A^g_C(S).
$$
This implies
\begin{multline}\label{eq3}
\left\langle T_{ts}c-P^g_{A^g_C(S)}(T_{ts}c),\nabla g(y)-\nabla g(u) \right\rangle\\
\leq \left\langle T_{ts}c-P^g_{A^g_C(S)}(T_{ts}c),\nabla g(P^g_{A^g_C(S)}(T_{ts}c))-\nabla g(u) \right\rangle.
\end{multline}

Since $O_c$ is  bounded, it follows from condition (2) in the definition of a Bregman function and Lemma \ref{LemmaBregProj} that
$\{P^g_{A^g_C(S)}(T_{ts}c): s\in S\}$  is also bounded.  Hence,
there exists $M>0$ such that $\|T_{ts}c\|+\|P^g_{A^g_C(S)}(T_{ts}c)\|\leq M$ for all $t, s\in S$.
It follows from \eqref{eq3} that
$$
\left\langle T_{ts}c-P^g_{A^g_C(S)}(T_{ts}c),\nabla g(y)-\nabla g(u) \right\rangle\leq M \| \nabla g(P^g_{A^g_C(S)}(T_{ts}c))-\nabla g(u)\|.
$$
Since $\mu$ is a mean on $X$, we have
\begin{equation*}
\begin{array}{rl}
\mu_t\left\langle T_{ts}c-P^g_{A^g_C(S)}(T_{ts}c),\nabla g(y)-\nabla g(u) \right\rangle&\leq\displaystyle\sup_{t\in S}\left\langle T_{ts}c-P^g_{A^g_C(S)}(T_{ts}c),\nabla g(y)-\nabla g(u) \right\rangle\\
&\leq M \displaystyle\sup_{t\in S}\| \nabla g(P^g_{A^g_C(S)}(T_{ts}c))-\nabla g(u)\|.
\end{array}
\end{equation*}
Since $\mu$ is right invariant, this implies
\begin{equation*}
\mu_t\left\langle T_{t}c-P^g_{A^g_C(S)}(T_{ts}c),\nabla g(y)-\nabla g(u) \right\rangle\leq M \displaystyle\sup_{t\in S}\| \nabla g(P^g_{A^g_C(S)}(T_{ts}c))-\nabla g(u)\|.
\end{equation*}
Because $P^g_{A^g_C(S)}(T_{s}c)$ converges strongly (in the index $s$) to $u$ and because
 $\nabla g$ is norm-to-norm continuous,  $\| \nabla g(P^g_{A^g_C(S)}(T_{ts}c))-\nabla g(u)\|\to 0$ uniformly on $t\in S$ when $s\to \infty$. Then, by taking the limit in $s$ we have $$\mu_t\left\langle T_{t}c-u,\nabla g(y)-\nabla g(u) \right\rangle\leq 0.$$
In other words,
$$
\left\langle z_\mu-u,\nabla g(y)-\nabla g(u) \right\rangle\leq 0, \quad\forall y\in A^g_C(S).
$$
Let $y=z_\mu$, we have $$0\geq\left\langle z_\mu-u,\nabla g(z_\mu)-\nabla g(u) \right\rangle=D_g(z_\mu,u)+D_g(u,z_\mu).$$
Hence $u=z_\mu$.
\end{proof}

\section{Bregman subspaces associated to the action}\label{sectionBregSubspace}

In this section, we study under what conditions on an action $(s,x)\mapsto T_s x$,
the coefficient functions in \eqref{BregCoefunction} belong to some classical function spaces  defined on $S$.
Since there are many well established amenability conditions concerning several classical function spaces on
semitopological semigroups, see, e.g.,
\cite{LauTaka95, LauZhang08, LauZhang16},
we can guarantee the existence of a LIM for the Bregman subspace $X$ and thus  utilize the results in section \ref{section3}.

Let $\operatorname{AP(S)}$ (resp.\ $\operatorname{WAP(S)}$) be the set of all \textit{almost periodic}
(resp.\ \textit{weakly almost periodic}) functions on $S$, i.e., all those functions $f\in\operatorname{CB(S)}$
such that its right orbit $\{r_tf: t\in S\}$ is precompact  in the norm (resp.\ weak) topology of $\operatorname{CB(S)}$.
Let $\operatorname{RUC(S)}$ be the set of all \textit{right uniformly continuous} on $S$, i.e.,
all those functions $f\in\operatorname{CB(S)}$ such that the map $s\to r_sf$ from $S$ into $\operatorname{CB(S)}$
is continuous when $\operatorname{CB(S)}$ is
equipped with the uniform norm topology. In general, $\operatorname{AP(S)}\subset\operatorname{WAP(S)}\subset\operatorname{CB(S)}$
and $\operatorname{AP(S)}\subset\operatorname{RUC(S)}\subset\operatorname{CB(S)}$.
When $S$ is compact, we have $\operatorname{AP}(S)=\operatorname{RUC}(S)\subset\operatorname{WAP(S)}\subset\operatorname{CB}(S)$.
The reader can see \cite{BJM88,Lau73Rocky,LauZhang08,MIT70} for more discussion about these function spaces and their amenability.

Let $c\in C$.
We call its orbit $O_c=\{T_sc: s\in S\}$ a \emph{continuous orbit} if the map $s\to T_sc$ is continuous in norm topology,
  a
\emph{nonexpansive orbit} if $\|T_sT_{t}c-T_sT_{t'}c\|\leq\|T_tc-T_{t'}c\|$ for all $s, t, t'\in S$,
and  a \emph{precompact orbit} if $O_c$ is norm precompact.  Note that $O_c$ will be precompact in the weak topology of the reflexive
Banach space $E$ whenever it is bounded.

\begin{Proposition}\label{BregSpaceAss}
	Let $(s,x)\mapsto T_s x$ be an action of a semitopological semigroup $S$ on a nonempty subset $C$ of a reflexive Banach space $E$.
Assume that there is a point $c\in C$ with a  bounded orbit $O_c$. Let $X$ be the Bregman subspace of
$\ell^\infty(S)$ associated to $O_c$ for the Bregman distance $D_g$. Then
	\begin{itemize}
		\item [(i)] $X\subset\operatorname{CB(S)}$ if $c$ has a continuous orbit;
		\item [(ii)] $X\subset\operatorname{RUC(S)}$ if $c$ has a continuous
nonexpansive orbit and the Bregman function $g$ is Lipschitz continuous on $O_c$;
		\item [(iii)] $X\subset\operatorname{AP(S)}$ if $c$ has a nonexpansive precompact orbit and $g$ is Lipschitz continuous on $O_c$;
		\item [(iv)] $X\subset\operatorname{WAP(S)}$ if each $T_s$ is weak-to-weak continuous and $g$ is weakly continuous on  $O_c$. 	
	\end{itemize}
\end{Proposition}
\begin{proof}
	We show that, for each $x\in E$ and each $x^*\in E^*$, the  coefficient functions
	$$s\to f(s)=\left\langle T_sc,x^*\right\rangle\quad \text{and}\quad s\to h(s)=D_g(T_sc,x)$$
	 are in the corresponding function spaces. The proof for $s\to D_g(x,T_sc)$ is similar.

	{(i)} Trivial.
	
	{(ii)} Let $t_\lambda\rightarrow t$.  For each $s\in S$, we have
	\begin{equation*}
	\begin{array}{rl}
	|r_{t_\lambda}f(s)-r_tf(s)|&=|\left\langle T_{st_\lambda}c-T_{st}c,x^*\right\rangle|\\
	&\leq\|x^*\|. \|T_sT_{t_\lambda}c-T_sT_tc\|\\
	&\leq \|x^*\|. \|T_{t_\lambda}c-T_tc\|.
	\end{array}
	\end{equation*}
	Hence $\|r_{t_\lambda}f-r_tf\|\leq \|x^*\|. \|T_{t_\lambda}c-T_tc\|\rightarrow 0$ and thus $f\in\operatorname{RUC(S)}$.
	
Let $K>0$ be the Lipschitz constant of $g$.  Then $\|\nabla g(y)\|\leq K$ for all $y\in E$.  Observe
	\begin{equation*}
	\begin{array}{rl}
	|r_{t_\lambda}h(s)-r_th(s)|&=|D_g(T_{st_\lambda}c,x)-D_g(T_{st}c,x)|\\
	&=|g(T_sT_{t_\lambda}c)-g(T_sT_tc)-\left\langle T_sT_{t_\lambda}c-T_sT_tc,\nabla g(x)\right\rangle|\\
	&\leq 2K\|T_sT_{t_\lambda}c-T_sT_tc\|\\
	&\leq 2K\|T_{t_\lambda}c-T_tc\|.
	\end{array}
	\end{equation*}
	Hence $\|r_{t_\lambda}h-r_th\|\leq 2K\|T_{t_\lambda}c-T_tc\|\rightarrow 0$, and thus $h\in\operatorname{RUC(S)}$.
	
	{(iii)} It follows from \cite[Lemma 3.1]{Lau73Rocky} that $f\in \operatorname{AP(S)}$.
	
	We  show that the map $L:O_c\to\operatorname{CB(S)}$ given by $L(z)=\varphi_z$ is norm-to-norm continuous where
$\varphi_z(s)=D_g(T_sz,x)$. Indeed, let $\{z_n\}$ be a sequence converging to $z$ in $O_c$ in norm. For each $s\in S$, since $g$ is
$K$-Lipschitz we have
\begin{equation*}
\begin{array}{rl}
|D_g(T_{s}z_n,x)-D_g(T_{s}z,x)|&=|g(T_sz_n)-g(T_sz)-\left\langle T_sz_n-T_sz,\nabla g(x)\right\rangle|\\
&\leq 2K\|T_sz_n-T_sz\|\leq 2K\|z_n-z\|.
\end{array}
\end{equation*}
Hence $L$ is norm-to-norm continuous. Since $O_c$ is norm precompact, so is $L({O_c})$.
On the other hand, $r_th(s)=h(st)=D_g(T_s(T_tc),x)$ for each $t\in S$. Hence $\{r_th: t\in S\}$
is contained in the precompact subset $L({O_c})$ of $\operatorname{CB(S)}$. In particular, $h\in\operatorname{AP(S)}$.	

	{(iv)} As in (iii) we define the map $L:O_c\to \operatorname{CB(S)}$ by
$$
z\mapsto D_g(T_sz,x) = g(T_sz) - g(x) - \langle T_sz -x, \nabla(x)\rangle.
$$
Since both $g$ and $T_s$ are weak-weak continuous, so is $L$.  Consequently, $L(O_c)$ is weakly precompact.
Because $\{r_th: t\in S\}\subset L(O_c)$,  it is also weakly precompact in $\operatorname{CB(S)}$.
In particular, $h\in\operatorname{WAP(S)}$.
	In similar manner,  we see that $f\in\operatorname{WAP(S)}$.	
\end{proof}

\section{Further discussion}\label{problems}

In section \ref{section3}, we have seen that a common fixed point exists if $X$ has a $\operatorname{LIM}$. But,
sometimes we only have an \emph{approximate left invariant mean} instead, i.e., there is a net $\{\mu_\lambda\}$ of means on $X$ such that $\mu_\lambda (f-l_sf)\to0$  for every $s\in S$ and every $f\in X$.  In this situation, we ask the following questions.

\begin{itemize}
	\item [\textbf{Q1:}] Do we have a common fixed point?
	\item [\textbf{Q2:}] When a fixed point exist, how do we locate it?
\end{itemize}

\begin{Proposition}
Let $C$ be a nonempty closed convex subset of a reflexive Banach space $E$.
Let $(s,x)\mapsto T_s x$ be an action of a semitopological semigroup $S$ on $C$ with a bounded continuous orbit $O_c$.
	Let $\{\mu_\lambda\}$ be an approximate left invariant mean on the Bregman subspace $X$ associated to
$O_c$. Then
	\begin{itemize}
		\item [(i)] any weak* cluster point $\mu$ of $\{\mu_\lambda\}$ is a left invariant mean and $z_{\mu_\lambda}\to  z_\mu$ weakly;
		\item [(ii)] if $\mu_\lambda\rightarrow \mu$ in norm then $z_{\mu_\lambda}\rightarrow z_\mu$ in norm.
	\end{itemize}
\end{Proposition}

\begin{proof}
{(i)} We note that there always exists a weak* cluster point of $\{\mu_\lambda\}$ since the closed unit ball of $X^*$ is weak* compact.
It follows similarly as \cite[page 883]{Taka92} that $\mu$ is a left invariant mean on $X$. By the definition of $z_\mu$
in \eqref{PsiFunctional}, $z_{\mu_\lambda}\to  z_\mu$ weakly.
	
{(ii)}
	If $\mu_\lambda$ converges to $\mu$ in norm then
	\begin{equation*}
	\begin{array}{rl}
	\|z_{\mu_\lambda}-z_\mu\|&=\displaystyle\sup_{\|x^*\|\leq 1}|\left\langle z_{\mu_\lambda},x^*\right\rangle-\left\langle z_\mu,x^*\right\rangle|\\
	&=\displaystyle\sup_{\|x^*\|\leq 1}|{\mu_\lambda}_t\left\langle T_tc,x^*\right\rangle-\mu_t\left\langle T_tc,x^*\right\rangle|\\
	&\leq M \| \mu_\lambda-\mu\|,
	\end{array}
	\end{equation*}
	where $M=\displaystyle\sup_{t\in S} \|T_tc\|$. This implies $z_{\mu_\lambda}$ converges to $z_\mu$ in norm.
\end{proof}

We end this paper with an open problem about a possible variance of our results.
Let $X$ be a subspace of $\ell^\infty(S)$ containing all constant functions. A real valued function $\mu$ on $X$ is called a \textit{submean}, see e.g. \cite{LauTaka95,LauZhang2018Annales}, if
\begin{itemize}
	\item [(1)] $\mu(f+g)\leq \mu(f)+\mu(g)$ for all $f,g\in X$;
	\item [(2)] $\mu(\alpha f)=\alpha\mu(f)$ for all $f\in X$ and $\alpha\geq 0$;
	\item [(3)] $f,g\in X$ and $f\leq g$ imply $\mu(f)\leq\mu(g)$;
	\item [(4)] $\mu(c)=c$ for every constant function $c$.
\end{itemize}
By \cite[proposition 3.6]{LauTaka95}, if $S$ is a left reversible semitopological semigroup, then $\operatorname{CB(S)}$ always has a left invariant submean, while, in general, it might not have any left invariant mean.
\begin{question}
	Do we have similar results as in the Theorems \ref{mainThm}, \ref{mainThm2} and \ref{mainthm3} when $X$ has a  (left) invariant submean?
\end{question}

Here is the difficulty.  Since we do not have the linearity of a submean $\mu$, the functional on $E^*$ defined by
 $x^*\mapsto\mu_s\left\langle T_sc,x^*\right\rangle$ in \eqref{PsiFunctional}  is not linear.
 Therefore, we are unable to define the $\mu-$barycenter $z_\mu$
  by the formula $\mu_s\left\langle T_sc,x^*\right\rangle=\left\langle z_\mu,x^*\right\rangle$ for all $x^*\in E^*$.

\section*{Acknowledgement}

This research is supported by the Taiwan MOST grant 108-2115-M-110-004-MY2.

\end{document}